\documentclass{amsart}

\usepackage{ adjustbox }
\usepackage{ amsthm }
\usepackage{ amssymb }
\usepackage{ amsmath }
\usepackage{ array }
\usepackage{ booktabs }
\usepackage{ cite }
\usepackage{ color }
\usepackage{ enumitem }
\usepackage{ latexsym }
\usepackage{ url }
\usepackage[all]{ xy }

\theoremstyle{definition}

\newtheorem{definition}{Definition}[section]

\newtheorem{conjecture}[definition]{Conjecture}
\newtheorem{example}[definition]{Example}

\newtheorem{remark}[definition]{Remark}

\theoremstyle{plain}

\newtheorem{lemma}[definition]{Lemma}
\newtheorem{proposition}[definition]{Proposition}
\newtheorem{theorem}[definition]{Theorem}

\newcommand{\fiveternary}{\mathcal{FT}}

\makeatletter
\renewcommand{\tocsection}[3]{%
\indentlabel{\@ifnotempty{#2}{\makebox[1.50em][l]{\ignorespaces#1#2.}}}#3}
\renewcommand{\tocsubsection}[3]{%
\indentlabel{\@ifnotempty{#2}{\hspace*{1.50em}\makebox[2.25em][l]{\ignorespaces#1#2.}}}#3}
\renewcommand{\tocsubsubsection}[3]{%
\indentlabel{\@ifnotempty{#2}{\hspace*{3.75em}\makebox[3.00em][l]{\ignorespaces#1#2.}}}#3}
\makeatother

\allowdisplaybreaks

\begin{document}

\title{The Veronese square of the dendriform operad}

\author{Murray R. Bremner}

\address{Department of Mathematics and Statistics,
University of Saskatchewan,
Canada}

\email{bremner@math.usask.ca}



\begin{abstract}
Veronese powers of operads were introduced in 2020 By Dotsenko, Markl, and Remm \cite{DMR}. 
The $m$-th Veronese power of a weight-graded operad $\mathcal{V}$ is the suboperad 
$\mathcal{V}^{[m]}$ generated by the operations of weight $m$.  
If $\mathcal{V}$ is generated by binary operations and governs the variety $\mathbf{V}$ 
of algebras, this gives a natural definition of the concept of $(m{+}1)$-ary $\mathbf{V}$-algebras.  
In particular, the Veronese square ($m=2$) corresponds to ternary algebras.  
We choose five generating operations for the Veronese square of the dendriform operad.
We represent the dendriform operad as a suboperad of the Rota-Baxter operad,
and express the quadratic relations satisfied by the generating operations 
as the kernel of a rewriting map. 
We use combinatorics of monomials and computational linear algebra to determine the kernel.
We obtain 33 linearly independent quadratic relations satisfied by the Veronese square. 
\end{abstract}

\maketitle

{\footnotesize\tableofcontents}


\section{Introduction}
\label{sectionintroduction}


\subsection{Summary of results}

Veronese powers of operads were introduced in 2020 by Dotsenko, Markl, and Remm \cite{DMR}. 
The $m$-th Veronese power of a weight-graded operad $\mathcal{P}$ is the suboperad $\mathcal{P}^{[m]}$
generated by the operations of weight $m$.  
If $\mathbf{V}$ is a variety of of binary algebras (associative, Lie, Jordan, etc.) governed by 
the operad $\mathcal{V}$, then algebras over the operad $\mathcal{V}^{[m]}$ may be regarded as
$(m{+}1)$-ary $\mathbf{V}$-algebras.
In particular, the Veronese square $\mathcal{V}^{[2]}$ provides a natural setting
for ternary $\mathbf{V}$-algebras, also called $\mathbf{V}$-triple systems,
including the classical associative, Lie and Jordan triple systems. 

We focus on the dendriform operad $\mathcal{D}$ and its embedding into 
the (noncommutative) Rota-Baxter operad $\mathcal{RB}$.
We represent $\mathcal{D}$ as the quotient operad $\mathcal{BB}/\mathcal{I}$ where
$\mathcal{BB}$ is the free nonsymmetric operad generated by two binary operations and
$\mathcal{I}$ is the ideal generated by the dendriform relations.
We choose generating operations for the Veronese square $\mathcal{D}^{[2]}$;
these are 5 elements of $\mathcal{BB}(3)$ which are linearly independent modulo $\mathcal{I}(3)$.
We then determine the 33 linearly independent quadratic relations (arity 5) satisfied 
by the generating operations.

We express the relations satisfied by the generating operations as the kernel of 
a rewriting morphism whose domain is the free nonsymmetric operad $\fiveternary$ generated by 
5 ternary operations and whose codomain is the quotient operad 
$\mathcal{RB} \cong \mathcal{UB}/\mathcal{J}$.
The nonsymmetric operad $\mathcal{UB}$ is generated by one unary operation and 
one associative binary operation, and $\mathcal{J}$ is 
the ideal generated by the Rota-Baxter relation.


We work throughout over the field $\mathbb{Q}$ of rational numbers, but it will be clear
that our results hold over an arbitrary field of characteristic 0.
All the operads we consider will be nonsymmetric and arity-graded unless otherwise specified.


\subsection{Contents of this paper}

In Section \ref{sectionpreliminaries} we recall basic results
about the dendriform operad, the Rota-Baxter operad, and the embedding
$\mathcal{D} \hookrightarrow \mathcal{RB}$.

In Section \ref{sectionoperations} we introduce the 5 generating operations for 
the Veronese square of the dendriform operad, and the free nonsymmetric operad
$\fiveternary$ generated by 5 ternary operations which is the domain of 
the rewriting morphism.

In Section \ref{sectionconsequences} we recall the Rota-Baxter operad, and introduce
operator monomials, their enumeration, an algorithm for generating them, and the natural
total order induced by the lex order on Dyck words.
We then discuss consequences of the Rota-Baxter relation determined by sequences
of partial compositions with both the unary and binary operations, and an algorithm
for generating these consequences.
We then define the matrix of consequences of the Rota-Baxter relation.

In Section \ref{sectionrewriting} we discuss the rewriting morphism 
$r \colon \fiveternary \to \mathcal{RB}$, and its restriction to arity 5
which converts ternary monomials into operator monomials of arity 5 and multiplicity 4
(the number of occurrences of the operator).
We define the rewriting matrix which collects this information into a suitable form
for further computations.

In Section \ref{sectionmaintheorem} we state and prove our main theorem:
that every quadratic relation satisfied by the 5 generating operations for 
the Veronese square of the dendriform operad is a linear combination of 33
basis relations which are explicitly presented.


\section{Preliminaries}
\label{sectionpreliminaries}

For basic information about operads, see 
Markl, Shnider, and Stasheff \cite{MSS} (which focusses on applications),
Loday and Vallette \cite{LV} (a comprehensive theoretical monograph), and 
the author and Dotsenko \cite{BD} (for the algorithmic aspects).
In particular, for nonsymmetric operads see \cite[Section 5.9]{LV} and \cite[Chapter 3]{BD}.
We recall the definition of nonsymmetric operad in terms of partial compositions.

\begin{definition}
A \emph{nonsymmetric operad} 
$\mathcal{P} = \{ \mathcal{P}(n) \}_{n \ge 0}$ is a collection of vector spaces 
together with an element $I \in \mathcal{P}(1)$ and maps (\emph{partial compositions})
\[
\circ_i \colon \mathcal{P}(m) \otimes \mathcal{P}(n) \longrightarrow \mathcal{P}(m+n-1),
\qquad
\alpha \otimes \beta \longmapsto \alpha \circ_i \beta,
\]
which satisfy the following axioms for all 
$\alpha \in \mathcal{P}(n)$, 
$\beta \in \mathcal{P}(m)$, 
$\gamma \in \mathcal{P}(r)$:
\begin{itemize}
\item
unit axiom:
$I \circ_1 \alpha = \alpha \circ_i I = \alpha$ 
for $1 \le i \le n$
\item
sequential axiom:
$( \alpha \circ_i \beta ) \circ_j \gamma = \alpha \circ_i ( \beta \circ_{j-i+1} \gamma )$
for $i \le j \le i+m-1$
\item
parallel axiom:
\[
( \alpha \circ_i \beta ) \circ_j \gamma = 
\begin{cases}
\;
( \alpha \circ_{j-m+1} \gamma ) \circ_i \beta &\quad \text{for} \; i+m \le j \le n+m-1
\\
\;
( \alpha \circ_j \gamma ) \circ_{i+r-1} \beta &\quad \text{for} \; 1 \le j \le i-1
\end{cases}
\]
\end{itemize}
\end{definition}

\begin{definition}
\label{def-dendriform}
Loday \cite[Section 5]{L}; Loday and Ronco \cite[Theorem 3.5]{LR}.
The \emph{dendriform operad} is the nonsymmetric operad $\mathcal{D}$ generated by 
two binary operations $x \prec y$ and $x \succ y$, 
called the \emph{left} and \emph{right} operations, satisfying 
\begin{align*}
( x \succ y ) \prec z &\equiv x \succ ( y \prec z ),
\\
( x \prec y ) \prec z &\equiv x \prec ( y \prec z ) + x \prec ( y \succ z ),
\\
x \succ ( y \succ z ) &\equiv ( x \succ y ) \succ z + ( x \prec y ) \succ z.
\end{align*}
We will identify the relation $v \equiv w$ with the relation $v - w \equiv 0$.
(Eilenberg and Mac Lane \cite[Section 18]{EM} came very close to defining 
dendriform algebras in 1953.
I thank L\'opez et al.~\cite[\S 1.1]{LPR} for this reference.)
\end{definition}

\begin{proposition}
\emph{Loday \cite[Theorem 5.8 and Section A.1]{L}}.
The dimensions of the homogeneous subspaces of the dendriform operad are the Catalan numbers:
\[
\dim \mathcal{D}(n) = \frac{1}{n+1} \binom{2n}{n}
\]
\end{proposition}

We will use later the fact that $\dim \mathcal{D}(3) = 5$ and that
a basis consists of the monomials on the right sides of the dendriform relations
in Definition \ref{def-dendriform}.

\begin{definition}
The (noncommutative) \emph{Rota-Baxter operad} $\mathcal{RB}$ is the nonsymmetric operad
generated by a unary operation $U$ denoted $x \mapsto U(x)$, and an associative binary 
operation $B$ denoted $(x,y) \mapsto xy$, satisfying
\[
U(x)U(y) = U(U(x)y) + U(xU(y)).
\]
\end{definition}

For a brief introduction to Rota-Baxter algebras, see Guo \cite{G2009};
the same author has written a monograph on this topic \cite{G2012}.
Aguiar \cite{A} was the first to notice that every Rota-Baxter algebra
has a natural structure of dendriform algebra.

\begin{proposition}
\emph{(Embedding Theorem)}
Let the map $\epsilon\colon \mathcal{D} \longrightarrow \mathcal{RB}$
be defined by
\[
\epsilon\colon x \prec y \; \longmapsto \; xR(y),   
\qquad
\epsilon\colon x \succ y \; \longmapsto \; R(x)y.
\]
Then $\epsilon$ extends to an injective morphism of operads.
\end{proposition}

\begin{proof}
See Chen and Mo \cite{CM} for the algebra version, and Gubarev and Kolesnikov \cite{GK} 
for a more general operadic result.
\end{proof}

This allows us to transfer computations in $\mathcal{D}$ 
to its isomorphic copy $\epsilon(\mathcal{D}) \subset \mathcal{RB}$.

\begin{definition}
\cite[Definition 3.6]{DMR}
The $m$-th \emph{Veronese power} of a weight-graded operad $\mathcal{P}$ 
is the suboperad $\mathcal{P}^{[m]}$ generated by all operations of weight $m$.
In particular, the \emph{second Veronese power} (or \emph{Veronese square}) 
$\mathcal{P}^{[2]}$ is the suboperad generated by all compositions of 
two generating operations of $\mathcal{P}$.
\end{definition}

Suppose that the operad $\mathcal{V}$ is generated by binary operations and that
it governs the variety $\mathbf{V}$ of algebras.
The notion of the Veronese square of $\mathcal{V}$ provides a natural operadic setting 
for the study of $\mathbf{V}$-triple systems which are algebras over $\mathcal{V}^{[2]}$.
Well-known examples are
associative triple systems \cite{Lister1}, 
Lie triple systems \cite{Lister2} and 
Jordan triple systems \cite{Neher}.
Some other varieties of triple systems are alternative \cite{Loos}, Leibniz \cite{BSO},
Poisson \cite{BE2}, and tortkara \cite{B2}.


\section{Generators for the Veronese square $\mathcal{D}^{[2]}$}
\label{sectionoperations}

\begin{definition}
\label{defBB}
Let $\mathcal{BB}$ denote the free nonsymmetric operad generated by two binary operations
$\prec$ and $\succ$
(the context will make clear whether we mean these operations or the dendriform operations).
Then $\mathcal{D} \cong \mathcal{BB} / \mathcal{I}$, where
$\mathcal{I} \subset \mathcal{BB}$ is the ideal generated by the relations defining 
the dendriform operad.
\end{definition}

The homogeneous component $\mathcal{BB}(3)$ has a standard basis consisting of the 
8 monomials appearing in the dendriform relations.
These relations form a basis for the homogeneous component $\mathcal{I}(3)$.
The Veronese square $\mathcal{D}^{[2]}$ is the suboperad of $\mathcal{D}$ generated by 
$\mathcal{D}(3)$.
Since $\dim \mathcal{D}(3) = 5$, as generating operations for $\mathcal{D}^{[2]}$
we may take (the cosets modulo $\mathcal{I}(3)$ of) any 5 elements of $\mathcal{BB}(3)$
which are linearly independent modulo $\mathcal{I}(3)$.

\begin{lemma}
\label{5operations}
The 5 monomials on the right sides of the dendriform relations are linearly independent modulo 
$\mathcal{I}(3)$:
\[
x \succ ( y \prec z ), \quad
x \prec ( y \prec z ), \quad
x \prec ( y \succ z ), \quad
( x \succ y ) \succ z, \quad
( x \prec y ) \succ z.
\]
\end{lemma}

\begin{remark}
We choose these 5 monomials for the generating operations of $\mathcal{D}^{[2]}$ 
because of their symmetries:
if we replace each operation by the opposite of the other operation,
then the set of 5 monomials does not change, except that we need to use inner associativity
for $x \succ ( y \prec z )$.
\end{remark}

\begin{definition}
We write $\fiveternary$ for the free nonsymmetric operad generated by 5 ternary operations
denoted $\omega_1, \dots, \omega_5$.
\end{definition}

The operad $\fiveternary$, and in particular its homogeneous component $\fiveternary(5)$, 
will be the domain of the rewriting morphism discussed in Section \ref{sectionrewriting}.
At that point, we will identify the 5 generators of $\fiveternary$ with the 5 monomials
of Lemma \ref{5operations}.
We will then use the embedding of $\mathcal{D}$ into $\mathcal{RB}$ to define the rewriting 
morphism, whose kernel will consist of the quadratic relations satisfied by $\mathcal{D}^{[2]}$.


\section{Consequences of the Rota-Baxter relation}
\label{sectionconsequences}

\begin{definition}
We write $\mathcal{UB}$ for the nonsymmetric operad generated by one unary operation $U$
and one (noncommutative) associative binary operation $B$.
Basis monomials of $\mathcal{UB}$ will be called \emph{operator monomials}.
This operad is bigraded: $\mathcal{UB}(p,q)$ denotes the homogeneous component
spanned by the monomials of arity $p$ and multiplicity $q$
(the number of occurrences of $U$).
The (noncommutative) \emph{Rota-Baxter operad} is the quotient 
$\mathcal{RB} = \mathcal{UB}/\mathcal{J}$ where $\mathcal{J} \subset \mathcal{UB}$ 
is the ideal generated by the \emph{Rota-Baxter relation};
we will denote its left side by $R$:
\[
U(x)U(y) - U(U(x)y) - U(xU(y)) \equiv 0.
\]
\end{definition}

\begin{lemma}
For $p \ge 1$ and $q \ge 0$ we have
\[
\dim \, \mathcal{UB}(p,q) = \frac{1}{p+q} \binom{p+q}{p} \binom{p+q}{p-1}
\]
\end{lemma}

\begin{proof}
These are the well-known Narayana numbers \cite[Lemma 2.5]{BE}.
\end{proof}

\begin{example}
\label{UBdims}
We present a small table of $\dim \, \mathcal{UB}(p,q)$:
\[
\begin{array}{r|rrrrr}
p \backslash q & 0 & 1 & 2 & 3 & 4 \\ \midrule
1 & 1 & 1 & 1 & 1 & 1 \\ 
2 & 1 & 3 & 6 & 10 & 15 \\ 
3 & 1 & 6 & 20 & 50 & 105 \\ 
4 & 1 & 10 & 50 & 175 & 490 \\ 
5 & 1 & 15 & 105 & 490 & 1764
\end{array}
\]
\end{example}

Figure \ref{monomialalgorithm} presents an algorithm which generates all operator monomials 
up to a given arity and multiplicity, using $X$ as a generic argument symbol.
Each operator monomial is a list enclosed in brackets, 
and brackets are also used without ambiguity to indicate the action of the operator.
The algorithm works since any monomial of multiplicity $q \ge 1$
can be obtained from a monomial of multiplicity $q-1$ 
by enclosing a submonomial in operator brackets.

\begin{figure}
\begin{itemize}
\item
\texttt{maxp} := maximum arity of operator monomials to be generated
\item
\texttt{maxq} := maximum multiplicity of operator monomials to be generated
\item
for $p$ to \texttt{maxp} do ($p$ is the arity of the monomials)
  \begin{enumerate}
  \item
  (only monomial of multiplicity 0 is list of $p$ arguments)
  \item[]
  \texttt{monomials}[ $p$, 0 ] := [\,[$X$,\dots,$X$]\,]
  \item
  for $q$ to \texttt{maxq} do ($q$ is the multiplicity of the monomials)
    \begin{enumerate}
    \item
    \texttt{monomials}$[p,q]$ := $[\;]$ (the empty list)
    \item
    (loop through all monomials with one less operator)
    \item[]
    for $m$ in \texttt{monomials}$[ p, q{-}1 ]$ do
      \begin{enumerate}
      \item
      $k$ := length($m$) (the number of factors in the monomial $m$)
      \item
      (double loop through all submonomials of $m$)
      \item[]
      for $i$ to $k$ do for $j$ from $i$ to $k$ do
        \begin{itemize}
        \item[$\ast$]
        (add operator brackets around submonomial)
        \item[]
        $m'$ := [ $m_1$, \dots, $m_{i-1}$, [ $m_i$, \dots, $m_j$ ], $m_{j+1}$, \dots, $m_k$ ]
        \item[$\ast$]
        append $m'$ to \texttt{monomials}$[p,q]$
        \end{itemize}
      \end{enumerate}
    \item
    eliminate repetitions from \texttt{monomials}$[p,q]$
    \item
    sort \texttt{monomials}$[p,q]$ in lex order of corresponding Dyck words
    \end{enumerate}
  \end{enumerate}
\end{itemize}
\vskip -10pt
\caption{Algorithm for generating operator monomials}
\label{monomialalgorithm}
\end{figure}

\begin{definition}
A \emph{Dyck word} is a string of left and right parentheses which is \emph{balanced} 
in the sense that 
(i) 
the string contains an equal number of left and right parentheses, and
(ii) 
in every initial substring, the number of left parentheses is 
greater than or equal to the number of right parentheses.
A substring $(\,)$ is called a \emph{nesting}.
We define the \emph{lex order} on Dyck words $v$ and $w$ of the same length:
let $i$ be the least index for which $v_i \ne w_i$; 
$v$ precedes $w$ if and only if $v_i = {(}$ and $w_i = {)}$.
\end{definition}

There is a bijection between the basis operator monomials of $\mathcal{UB}(p,q)$ 
and the Dyck words of length $2(p+q)$ which contain $p$ nestings.
In one direction, the bijection may be computed as follows.
Given such a Dyck word, we first replace the $p$ nestings by the arguments
$x_1, \dots, x_p$ from left to right, and then insert the operator symbol $U$ 
immediately before every remaining left parenthesis. 
Since we consider only nonsymmetric operads, we may replace each of the arguments
$x_1, \dots, x_p$ by the generic argument symbol $*$ 
(corresponding to $X$ in the algorithm), since the subscript on $x_i$ 
merely indicates its position in the monomial.

To clarify the preceding discussion, Table \ref{table32} presents the 20 operator monomials
of arity 3 and multiplicity 2 in three different forms:
first, as operator monomials;
second, as the corresponding Dyck words (in lex order);
and third, as the $X$-lists generated by the algorithm in Figure \ref{monomialalgorithm}.
We write $U(U({*}))$ instead of $U^2({*})$.

\begin{table}
\small
$
\begin{array}{rrrr|rrrr}
 1 & U(U({*}{*}{*})) & ((()()())) & [[[X,X,X]]] &
 2 & U(U({*}{*}){*}) & ((()())()) & [[[X,X],X]] \\ 
 3 & U(U({*}{*})){*} & ((()()))() & [[[X,X]],X] &
 4 & U(U({*}){*}{*}) & ((())()()) & [[[X],X,X]] \\ 
 5 & U(U({*}){*}){*} & ((())())() & [[[X],X],X] &
 6 & U(U({*})){*}{*} & ((()))()() & [[[X]],X,X] \\ 
 7 & U({*}U({*}{*})) & (()(()())) & [[X,[X,X]]] &
 8 & U({*}U({*}){*}) & (()(())()) & [[X,[X],X]] \\ 
 9 & U({*}U({*})){*} & (()(()))() & [[X,[X]],X] &
10 & U({*}{*}U({*})) & (()()(())) & [[X,X,[X]]] \\ 
11 & U({*}{*})U({*}) & (()())(()) & [[X,X],[X]] &
12 & U({*})U({*}{*}) & (())(()()) & [[X],[X,X]] \\ 
13 & U({*})U({*}){*} & (())(())() & [[X],[X],X] &
14 & U({*}){*}U({*}) & (())()(()) & [[X],X,[X]] \\ 
15 & {*}U(U({*}{*})) & ()((()())) & [X,[[X,X]]] &
16 & {*}U(U({*}){*}) & ()((())()) & [X,[[X],X]] \\ 
17 & {*}U(U({*})){*} & ()((()))() & [X,[[X]],X] &
18 & {*}U({*}U({*})) & ()(()(())) & [X,[X,[X]]] \\ 
19 & {*}U({*})U({*}) & ()(())(()) & [X,[X],[X]] &
20 & {*}{*}U(U({*})) & ()()((())) & [X,X,[[X]]] 
\end{array}
$
\medskip
\caption{Operator monomials of arity 3 and multiplicity 2}
\label{table32}
\end{table}

We need to compute the consequences of the Rota-Baxter relation $R$ in higher arities and multiplicities.
In order to increase the arity, we perform partial composition with the binary operation $B$.
If $S \in \mathcal{UB}(p,q)$ then the following partial compositions produce
an element of $\mathcal{UB}(p+1,q)$:
\[
S \circ_i B \quad (1 \le i \le p), \qquad B \circ_j S \quad (1 \le j \le 2).
\]
To increase the multiplicity, we perform partial composition with the unary operation $U$.
The following partial compositions produce an element of $\mathcal{UB}(p,q+1)$:
\[
S \circ_i U \quad (1 \le i \le p), \qquad U \circ S.
\]
Since $U$ is unary we may omit the subscript on the last partial composition.
The consequences of $R$ in $\mathcal{UB}(p,q)$ form a spanning set for $\mathcal{J}(p,q)$.

The Rota-Baxter relation $R$ belongs to $\mathcal{UB}(2,2)$ and we will see later
that the codomain of the rewriting map is $\mathcal{UB}(5,4)$.
Starting with $R$, to obtain its consequences in $\mathcal{UB}(5,4)$
we need to perform 5 partial compositions, 3 with $B$ and 2 with $U$, 
which gives 10 possibilities corresponding to the sequences
\begin{align*}
&
UU\!BBB, \quad 
U\!BU\!BB, \quad
U\!BBU\!B, \quad
U\!BBBU, \quad
BUU\!BB, 
\\
&
BU\!BU\!B, \quad
BU\!BBU, \quad
BBUU\!B, \quad
BBU\!BU, \quad
BBBUU.
\end{align*}
Each of these sequences produces a subset of all consequences, but there is a great deal
of redundancy, because of both the associativity of $B$ and the parallel and sequential 
relations satisfied by partial compositions.
The number of consequences corresponding to each of the 10 sequences above is
\[
1080, \; 1440, \; 1800, \; 2160, \; 1920, \; 2400, \; 2880, \; 3000, \; 3600, \; 4320, \;
\]
for a total of 24600.
However, a nonredundant subset of these consequences contains only 1176 elements,
about $4.78\%$ of the total.

Figure \ref{pcalgorithm} presents an algorithm to perform all possible sequences of 
partial compositions to produce consequences in $\mathcal{UB}(p,q)$ of the Rota-Baxter relation 
$R$.

\begin{figure}
\begin{itemize}
\item
for $\texttt{uset}$ in all $q-2$ element subsets of $\{1,...,p+q-4\}$ do
  \begin{enumerate}
  \item
  \texttt{clist} := [ $R$ ] (Rota-Baxter relation)
  \item
  \texttt{arity} := 2,
  \texttt{multiplicity} := 2
  \item
  for $i$ to $p+q-4$ do
    \begin{enumerate}
    \item
    \texttt{clist1} := \texttt{clist},
    \texttt{clist2} := [\;] (the empty list)
    \item
    if $i \in \texttt{uset}$ then 
      \begin{enumerate}
      \item
      for $S$ in \texttt{clist1} do
        \begin{itemize}
        \item[$*$]
        for $k$ to \texttt{arity} do append $S \circ_k U$ to \texttt{clist2}
        \item[$*$]
        append $U \circ S$ to \texttt{clist2}
        \end{itemize}
      \item
      \texttt{clist} := \texttt{clist2}
      \item
      increment \texttt{multiplicity}
      \end{enumerate}
    else 
      \begin{enumerate}
      \item
      for $S$ in $\texttt{clist1}$ do
        \begin{itemize}
        \item[$*$]
        for $k$ to $\texttt{arity}$ do append $S \circ_k B$ to $\texttt{clist2}$
        \item[$*$]
        for $k$ to 2 do append $B \circ_k S$ to $\texttt{clist2}$
        \end{itemize}
      \item
      \texttt{clist} := \texttt{clist2}
      \item
      increment $\texttt{arity}$
      \end{enumerate}
    \end{enumerate}
    \item
    $\texttt{consequences}$[ $\texttt{uset}$ ] := $\texttt{clist}$
  \end{enumerate}
\end{itemize}
\vskip -10pt
\caption{Algorithm to generate consequences of RB relation}
\label{pcalgorithm}
\end{figure}

\begin{example}
We present 3 examples of consequences of the Rota-Baxter relation $R$
in $\mathcal{UB}(5,4)$ together with the corresponding sequences
of partial compositions:
\begin{align*}
&
( ( ( ( R \circ_2 U ) \circ_1 B ) \circ_1 U ) \circ_3 B ) \circ_3 B =
\\
&\qquad
U(U({*}){*})U(U({*}{*}{*})) - U(U(U({*}){*})U({*}{*}{*})) - U(U({*}){*}U(U({*}{*}{*}))),
\\
&
( U \circ ( B \circ_2 ( B \circ_1 ( R \circ_1 U ) ) ) ) \circ_2 B =
\\
&\qquad
U({*}U(U({*}{*}))U({*}){*}) - U({*}U(U(U({*}{*})){*}){*}) - U({*}U(U({*}{*})U({*})){*}), 
\\
&
( B \circ_2 ( U \circ ( U \circ ( B \circ_2 R ) ) ) ) \circ_1 B =
\\
&\qquad
{*}{*}U(U({*}U({*})U({*}))) - {*}{*}U(U({*}U(U({*}){*}))) - {*}{*}U(U({*}U({*}U({*})))). 
\end{align*}
\end{example}

\begin{definition}
Fix an arity $p$ and a multiplicity $q$.
Let $S \subseteq \mathcal{UB}(p,q)$ be a non-redundant set of consequences of $R$;
the order is not significant.
Let $M(p,q)$ be the set of operator monomials forming a lex-ordered basis of $\mathcal{UB}(p,q)$.
The \emph{matrix of consequences} $C(p,q)$ of the Rota-Baxter relation $R$ is the matrix 
with $|S|$ rows and $|M(p,q)|$ columns in which
the $ij$-entry is the coefficient of monomial $j$ in consequence $i$.
Clearly the row space of $C(p,q)$ equals $\mathcal{J}(p,q)$.
\end{definition}

The matrix of consequences is very sparse: each row contains only 3 nonzero entries ($\pm 1$),
and 120 columns are zero, corresponding to those operator monomials which do not appear 
in any consequence of the Rota-Baxter relation.
We have $\text{rank} \, C(p,q) = \dim \mathcal{J}(p,q)$
and so 
\[
\dim \, \mathcal{RB}(p,q) = \dim \, \mathcal{UB}(p,q) - \text{rank} \, C(p,q).
\]
For more information about the dimension of $\mathcal{RB}(p,q)$, see Guo and Sit \cite{GS2}.


\section{The rewriting morphism}
\label{sectionrewriting}

\begin{definition}
Since $\fiveternary$ is a free operad, we may define a morphism
with domain $\fiveternary$ on its 5 ternary generators $\omega_1, \dots, \omega_5$
which we also denote by $(x,y,z)_1$, \dots, $(x,y,z)_5$.
First consider the morphism $d\colon \fiveternary \to \mathcal{D}$ defined
by the following map $\mathcal{FT}(3) \to \mathcal{D}(3)$ which uses
the order of Lemma \ref{5operations}:
\begin{alignat*}{2}
\omega_1 = (x,y,z)_1 \; &\mapsto \;  x \succ ( y \prec z ), \\
\omega_2 = (x,y,z)_2 \; &\mapsto \;  x \prec ( y \prec z ), &\qquad
\omega_3 = (x,y,z)_3 \; &\mapsto \;  x \prec ( y \succ z ), \\
\omega_4 = (x,y,z)_4 \; &\mapsto \;  ( x \succ y ) \succ z, &\qquad
\omega_5 = (x,y,z)_5 \; &\mapsto \;  ( x \prec y ) \succ z.
\end{alignat*}
Composing $d$ with the embedding $\epsilon \colon \mathcal{D} \to \mathcal{RB}$
gives the morphism $r \colon \mathcal{FT} \to \mathcal{RB}$ which we call the \emph{rewriting morphism}:
\begin{alignat*}{2}
\omega_1 \; &\mapsto \;  U(x)yU(z), \\
\omega_2 \; &\mapsto \;  xU(yU(z)), &\qquad
\omega_3 \; &\mapsto \;  xU(U(y)z), \\
\omega_4 \; &\mapsto \;  U(U(x)y)z, &\qquad
\omega_5 \; &\mapsto \;  U(xU(y))z.
\end{alignat*}
In what follows we will regard the embedding $\epsilon$ as a morphism 
from $\mathcal{D}$ to $\mathcal{BB}$
followed by the projection $\mathcal{BB} \to \mathcal{BB}/\mathcal{J} \cong \mathcal{RB}$.
\end{definition}

For a ternary operation there are 3 association types in arity 5.
Since there are 5 ternary generators of $\mathcal{FT}$ we obtain a total of 
$3 \cdot 5^2 = 75$ basis monomials for $\mathcal{FT}(5)$.
In terms of both association types and partial compositions we have
\[
( ( {\ast} {\ast} {\ast} )_i {\ast} {\ast} )_j = \omega_j \circ_1 \omega_i, \quad
( {\ast} ( {\ast} {\ast} {\ast} )_i {\ast} )_j = \omega_j \circ_2 \omega_i, \quad
( {\ast} {\ast} ( {\ast} {\ast} {\ast} )_i )_j = \omega_j \circ_3 \omega_i.
\]
Table \ref{rewrite5} lists these 75 monomials using partial compositions
but in lex order of association types and operation subscripts:
the partial compositions $m = \omega_r \circ_s \omega_t$ are listed in
lex order of the triples $(s,t,r)$.
In the table, $i$ is the index of the 75 ternary monomials $m$ in lex order.
For each $m$, the table presents $r(m)$, and the index $j$ of $r(m)$ using the
lex order of operator monomials.
The monomials $r(m) \in \mathcal{BB}$ have arity 5 and multiplicity 4.
The quadratic relations satisfied by the 5 generators of $\mathcal{D}^{[2]}$
are the kernel of the restriction of the rewriting map to the domain $\fiveternary(5)$
with codomain $\mathcal{BB}(5,4)$.

\begin{table}
\tiny
\[
\begin{array}{rclcr|rclcr}
i & m & & r(m) & j & i & m & & r(m) & j \\ \midrule
 1 & \omega_1 \circ_1 \omega_1 &\longmapsto & U(U({*}){*}U({*})){*}U({*}) &  629 &
 2 & \omega_2 \circ_1 \omega_1 &\longmapsto & U({*}){*}U({*})U({*}U({*})) & 1262 \\ 
 3 & \omega_3 \circ_1 \omega_1 &\longmapsto & U({*}){*}U({*})U(U({*}){*}) & 1260 &
 4 & \omega_4 \circ_1 \omega_1 &\longmapsto & U(U(U({*}){*}U({*})){*}){*} &  218 \\ 
 5 & \omega_5 \circ_1 \omega_1 &\longmapsto & U(U({*}){*}U({*})U({*})){*} &  624 &
 6 & \omega_1 \circ_1 \omega_2 &\longmapsto & U({*}U({*}U({*}))){*}U({*}) &  872 \\ 
 7 & \omega_2 \circ_1 \omega_2 &\longmapsto & {*}U({*}U({*}))U({*}U({*})) & 1583 &
 8 & \omega_3 \circ_1 \omega_2 &\longmapsto & {*}U({*}U({*}))U(U({*}){*}) & 1581 \\ 
 9 & \omega_4 \circ_1 \omega_2 &\longmapsto & U(U({*}U({*}U({*}))){*}){*} &  343 &
10 & \omega_5 \circ_1 \omega_2 &\longmapsto & U({*}U({*}U({*}))U({*})){*} &  867 \\ 
11 & \omega_1 \circ_1 \omega_3 &\longmapsto & U({*}U(U({*}){*})){*}U({*}) &  813 &
12 & \omega_2 \circ_1 \omega_3 &\longmapsto & {*}U(U({*}){*})U({*}U({*})) & 1497 \\ 
13 & \omega_3 \circ_1 \omega_3 &\longmapsto & {*}U(U({*}){*})U(U({*}){*}) & 1495 &
14 & \omega_4 \circ_1 \omega_3 &\longmapsto & U(U({*}U(U({*}){*})){*}){*} &  318 \\ 
15 & \omega_5 \circ_1 \omega_3 &\longmapsto & U({*}U(U({*}){*})U({*})){*} &  808 &
16 & \omega_1 \circ_1 \omega_4 &\longmapsto & U(U(U({*}){*}){*}){*}U({*}) &  248 \\ 
17 & \omega_2 \circ_1 \omega_4 &\longmapsto & U(U({*}){*}){*}U({*}U({*})) &  662 &
18 & \omega_3 \circ_1 \omega_4 &\longmapsto & U(U({*}){*}){*}U(U({*}){*}) &  660 \\ 
19 & \omega_4 \circ_1 \omega_4 &\longmapsto & U(U(U(U({*}){*}){*}){*}){*} &   50 &
20 & \omega_5 \circ_1 \omega_4 &\longmapsto & U(U(U({*}){*}){*}U({*})){*} &  243 \\ 
21 & \omega_1 \circ_1 \omega_5 &\longmapsto & U(U({*}U({*})){*}){*}U({*}) &  407 &
22 & \omega_2 \circ_1 \omega_5 &\longmapsto & U({*}U({*})){*}U({*}U({*})) &  957 \\ 
23 & \omega_3 \circ_1 \omega_5 &\longmapsto & U({*}U({*})){*}U(U({*}){*}) &  955 &
24 & \omega_4 \circ_1 \omega_5 &\longmapsto & U(U(U({*}U({*})){*}){*}){*} &  100 \\ 
25 & \omega_5 \circ_1 \omega_5 &\longmapsto & U(U({*}U({*})){*}U({*})){*} &  402 &
26 & \omega_1 \circ_2 \omega_1 &\longmapsto & U({*})U({*}){*}U({*})U({*}) & 1223 \\ 
27 & \omega_2 \circ_2 \omega_1 &\longmapsto & {*}U(U({*}){*}U({*})U({*})) & 1490 & 
28 & \omega_3 \circ_2 \omega_1 &\longmapsto & {*}U(U(U({*}){*}U({*})){*}) & 1361 \\ 
29 & \omega_4 \circ_2 \omega_1 &\longmapsto & U(U({*})U({*}){*}U({*})){*} &  593 & 
30 & \omega_5 \circ_2 \omega_1 &\longmapsto & U({*}U(U({*}){*}U({*}))){*} &  802 \\ 
31 & \omega_1 \circ_2 \omega_2 &\longmapsto & U({*}){*}U({*}U({*}))U({*}) & 1256 & 
32 & \omega_2 \circ_2 \omega_2 &\longmapsto & {*}U({*}U({*}U({*}))U({*})) & 1566 \\ 
33 & \omega_3 \circ_2 \omega_2 &\longmapsto & {*}U(U({*}U({*}U({*}))){*}) & 1412 & 
34 & \omega_4 \circ_2 \omega_2 &\longmapsto & U(U({*}){*}U({*}U({*}))){*} &  618 \\ 
35 & \omega_5 \circ_2 \omega_2 &\longmapsto & U({*}U({*}U({*}U({*})))){*} &  853 & 
36 & \omega_1 \circ_2 \omega_3 &\longmapsto & U({*}){*}U(U({*}){*})U({*}) & 1246 \\ 
37 & \omega_2 \circ_2 \omega_3 &\longmapsto & {*}U({*}U(U({*}){*})U({*})) & 1549 & 
38 & \omega_3 \circ_2 \omega_3 &\longmapsto & {*}U(U({*}U(U({*}){*})){*}) & 1401 \\ 
39 & \omega_4 \circ_2 \omega_3 &\longmapsto & U(U({*}){*}U(U({*}){*})){*} &  611 & 
40 & \omega_5 \circ_2 \omega_3 &\longmapsto & U({*}U({*}U(U({*}){*}))){*} &  842 \\ 
41 & \omega_1 \circ_2 \omega_4 &\longmapsto & U({*})U(U({*}){*}){*}U({*}) & 1158 & 
42 & \omega_2 \circ_2 \omega_4 &\longmapsto & {*}U(U(U({*}){*}){*}U({*})) & 1369 \\ 
43 & \omega_3 \circ_2 \omega_4 &\longmapsto & {*}U(U(U(U({*}){*}){*}){*}) & 1295 & 
44 & \omega_4 \circ_2 \omega_4 &\longmapsto & U(U({*})U(U({*}){*}){*}){*} &  554 \\ 
45 & \omega_5 \circ_2 \omega_4 &\longmapsto & U({*}U(U(U({*}){*}){*})){*} &  736 & 
46 & \omega_1 \circ_2 \omega_5 &\longmapsto & U({*})U({*}U({*})){*}U({*}) & 1189 \\ 
47 & \omega_2 \circ_2 \omega_5 &\longmapsto & {*}U(U({*}U({*})){*}U({*})) & 1428 & 
48 & \omega_3 \circ_2 \omega_5 &\longmapsto & {*}U(U(U({*}U({*})){*}){*}) & 1320 \\ 
49 & \omega_4 \circ_2 \omega_5 &\longmapsto & U(U({*})U({*}U({*})){*}){*} &  568 & 
50 & \omega_5 \circ_2 \omega_5 &\longmapsto & U({*}U(U({*}U({*})){*})){*} &  761 \\ 
51 & \omega_1 \circ_3 \omega_1 &\longmapsto & U({*}){*}U(U({*}){*}U({*})) & 1245 & 
52 & \omega_2 \circ_3 \omega_1 &\longmapsto & {*}U({*}U(U({*}){*}U({*}))) & 1548 \\ 
53 & \omega_3 \circ_3 \omega_1 &\longmapsto & {*}U(U({*})U({*}){*}U({*})) & 1480 & 
54 & \omega_4 \circ_3 \omega_1 &\longmapsto & U(U({*}){*})U({*}){*}U({*}) &  658 \\ 
55 & \omega_5 \circ_3 \omega_1 &\longmapsto & U({*}U({*}))U({*}){*}U({*}) &  953 & 
56 & \omega_1 \circ_3 \omega_2 &\longmapsto & U({*}){*}U({*}U({*}U({*}))) & 1254 \\ 
57 & \omega_2 \circ_3 \omega_2 &\longmapsto & {*}U({*}U({*}U({*}U({*})))) & 1564 & 
58 & \omega_3 \circ_3 \omega_2 &\longmapsto & {*}U(U({*}){*}U({*}U({*}))) & 1489 \\ 
59 & \omega_4 \circ_3 \omega_2 &\longmapsto & U(U({*}){*}){*}U({*}U({*})) &  662 & 
60 & \omega_5 \circ_3 \omega_2 &\longmapsto & U({*}U({*})){*}U({*}U({*})) &  957 \\ 
61 & \omega_1 \circ_3 \omega_3 &\longmapsto & U({*}){*}U({*}U(U({*}){*})) & 1251 & 
62 & \omega_2 \circ_3 \omega_3 &\longmapsto & {*}U({*}U({*}U(U({*}){*}))) & 1560 \\ 
63 & \omega_3 \circ_3 \omega_3 &\longmapsto & {*}U(U({*}){*}U(U({*}){*})) & 1486 & 
64 & \omega_4 \circ_3 \omega_3 &\longmapsto & U(U({*}){*}){*}U(U({*}){*}) &  660 \\ 
65 & \omega_5 \circ_3 \omega_3 &\longmapsto & U({*}U({*})){*}U(U({*}){*}) &  955 & 
66 & \omega_1 \circ_3 \omega_4 &\longmapsto & U({*}){*}U(U(U({*}){*}){*}) & 1230 \\ 
67 & \omega_2 \circ_3 \omega_4 &\longmapsto & {*}U({*}U(U(U({*}){*}){*})) & 1526 & 
68 & \omega_3 \circ_3 \omega_4 &\longmapsto & {*}U(U({*})U(U({*}){*}){*}) & 1465 \\ 
69 & \omega_4 \circ_3 \omega_4 &\longmapsto & U(U({*}){*})U(U({*}){*}){*} &  649 & 
70 & \omega_5 \circ_3 \omega_4 &\longmapsto & U({*}U({*}))U(U({*}){*}){*} &  944 \\ 
71 & \omega_1 \circ_3 \omega_5 &\longmapsto & U({*}){*}U(U({*}U({*})){*}) & 1237 & 
72 & \omega_2 \circ_3 \omega_5 &\longmapsto & {*}U({*}U(U({*}U({*})){*})) & 1537 \\ 
73 & \omega_3 \circ_3 \omega_5 &\longmapsto & {*}U(U({*})U({*}U({*})){*}) & 1472 &   
74 & \omega_4 \circ_3 \omega_5 &\longmapsto & U(U({*}){*})U({*}U({*})){*} &  653 \\ 
75 & \omega_5 \circ_3 \omega_5 &\longmapsto & U({*}U({*}))U({*}U({*})){*} &  948 
\end{array}
\]
\caption{The rewriting map $r\colon \fiveternary(5) \to \mathcal{UB}(5,4)$ on basis monomials}
\label{rewrite5}
\end{table}

\begin{definition}
The \emph{rewriting matrix} $W$ has size $75 \times 1764$;
its $ij$-entry is 0 unless operator monomial $j$ is the rewriting of ternary monomial $i$, 
in which case it is 1.
\end{definition}


\section{Main Theorem}
\label{sectionmaintheorem}

\begin{theorem}
\label{maintheorem}
Every quadratic relation satisfied by the operations $(-,-,-)_1$, \dots, $(-,-,-)_5$
is a linear combination of the following 33 relations (omitting $\equiv 0$):
\begin{align*}
& 
   ((v,w,x)_4,y,z)_2
-  (v,w,(x,y,z)_2)_4,
\\ 
& 
   ((v,w,x)_5,y,z)_2
-  (v,w,(x,y,z)_2)_5,
\\ 
& 
   ((v,w,x)_4,y,z)_3
-  (v,w,(x,y,z)_3)_4,
\\ 
& 
   ((v,w,x)_5,y,z)_3
-  (v,w,(x,y,z)_3)_5,
\\ 
& 
   ((v,w,x)_1,y,z)_1
+  ((v,w,x)_2,y,z)_1
-  (v,(w,x,y)_5,z)_1,
\\ 
& 
   (v,(w,x,y)_1,z)_3
+  (v,(w,x,y)_2,z)_3
-  (v,w,(x,y,z)_5)_3,
\\ 
& 
   ((v,w,x)_1,y,z)_4
+  ((v,w,x)_2,y,z)_4
-  (v,(w,x,y)_5,z)_4,
\\ 
& 
   ((v,w,x)_1,y,z)_1
+  ((v,w,x)_4,y,z)_1
-  (v,w,(x,y,z)_1)_4,
\\ 
& 
   (v,(w,x,y)_4,z)_2
+  (v,(w,x,y)_5,z)_2
-  (v,w,(x,y,z)_1)_3,
\\ 
& 
   ((v,w,x)_4,y,z)_5
+  ((v,w,x)_5,y,z)_5
-  (v,(w,x,y)_1,z)_4,
\\ 
& 
   ((v,w,x)_4,y,z)_5
+  (v,(w,x,y)_2,z)_4
-  (v,w,(x,y,z)_5)_4,
\\
& 
   (v,(w,x,y)_3,z)_1
-  (v,w,(x,y,z)_1)_1
-  (v,w,(x,y,z)_4)_1,
\\ 
& 
   (v,(w,x,y)_3,z)_2
-  (v,w,(x,y,z)_1)_2
-  (v,w,(x,y,z)_4)_2,
\\ 
& 
   ((v,w,x)_3,y,z)_5
-  (v,(w,x,y)_1,z)_5
-  (v,(w,x,y)_4,z)_5,
\\ 
& 
   ((v,w,x)_1,y,z)_2
-  (v,w,(x,y,z)_1)_1
-  (v,w,(x,y,z)_2)_1,
\\ 
& 
   ((v,w,x)_3,y,z)_2
-  (v,(w,x,y)_4,z)_2
-  (v,w,(x,y,z)_2)_3,
\\ 
& 
   (v,(w,x,y)_1,z)_2
-  (v,w,(x,y,z)_2)_3
-  (v,w,(x,y,z)_3)_3,
\\ 
& 
   ((v,w,x)_1,y,z)_5
-  (v,(w,x,y)_2,z)_4
-  (v,(w,x,y)_3,z)_4,
\\ 
& 
   (v,(w,x,y)_2,z)_1
-  (v,w,(x,y,z)_2)_1
-  (v,w,(x,y,z)_3)_1
-  (v,w,(x,y,z)_5)_1,
\\ 
& 
   ((v,w,x)_2,y,z)_2
-  (v,(w,x,y)_5,z)_2
-  (v,w,(x,y,z)_1)_2
-  (v,w,(x,y,z)_2)_2,
\\ 
& 
   (v,(w,x,y)_2,z)_2
-  (v,w,(x,y,z)_2)_2
-  (v,w,(x,y,z)_3)_2
-  (v,w,(x,y,z)_5)_2,
\\ 
& 
   ((v,w,x)_3,y,z)_3
-  (v,(w,x,y)_1,z)_3
-  (v,(w,x,y)_4,z)_3
-  (v,w,(x,y,z)_3)_3,
\\ 
& 
   (v,(w,x,y)_3,z)_3
+  (v,(w,x,y)_4,z)_3
+  (v,(w,x,y)_5,z)_3
-  (v,w,(x,y,z)_4)_3,
\\ 
& 
   ((v,w,x)_3,y,z)_4
+  ((v,w,x)_4,y,z)_4
+  ((v,w,x)_5,y,z)_4
-  (v,(w,x,y)_4,z)_4,
\\ 
& 
   ((v,w,x)_1,y,z)_4
+  ((v,w,x)_4,y,z)_4
+  (v,(w,x,y)_3,z)_4
-  (v,w,(x,y,z)_4)_4,
\\ 
& 
   ((v,w,x)_2,y,z)_5
-  (v,(w,x,y)_2,z)_5
-  (v,(w,x,y)_3,z)_5
-  (v,(w,x,y)_5,z)_5,
\\ 
& 
   ((v,w,x)_5,y,z)_5
+  (v,(w,x,y)_1,z)_5
+  (v,(w,x,y)_2,z)_5
-  (v,w,(x,y,z)_5)_5,
\\ 
& 
   ((v,w,x)_2,y,z)_1
+  ((v,w,x)_3,y,z)_1
+  ((v,w,x)_5,y,z)_1
-  (v,w,(x,y,z)_1)_5,
\\ 
& 
   (v,(w,x,y)_4,z)_1
+  (v,(w,x,y)_5,z)_1
-  (v,w,(x,y,z)_1)_4
-  (v,w,(x,y,z)_1)_5,
\\ 
& 
   ((v,w,x)_1,y,z)_3
-  (v,(w,x,y)_2,z)_1
+  (v,w,(x,y,z)_2)_1
-  (v,w,(x,y,z)_4)_1,
\\ 
& 
   (v,(w,x,y)_1,z)_1
-  (v,w,(x,y,z)_2)_4
-  (v,w,(x,y,z)_2)_5
-  (v,w,(x,y,z)_3)_4
\\
&\qquad
-  (v,w,(x,y,z)_3)_5,
\\ 
& 
   ((v,w,x)_2,y,z)_3
-  (v,(w,x,y)_2,z)_2
-  (v,(w,x,y)_2,z)_3
+  (v,(w,x,y)_4,z)_3
\\
&\qquad
+  (v,w,(x,y,z)_2)_2
-  (v,w,(x,y,z)_4)_2
-  (v,w,(x,y,z)_4)_3,
\\ 
& 
   ((v,w,x)_2,y,z)_4
+  ((v,w,x)_2,y,z)_5
-  ((v,w,x)_4,y,z)_4
-  (v,(w,x,y)_2,z)_5
\\
&\qquad
+  (v,(w,x,y)_4,z)_4
+  (v,(w,x,y)_4,z)_5
-  (v,w,(x,y,z)_4)_5.
\end{align*}
\end{theorem}

\begin{proof}
We give a proof based on computational linear algebra, using the computer algebra system Maple.
All of our calculations were done over the field $\mathbb{Q}$ of rational numbers,
except that lattice basis reduction is done over the ring $\mathbb{Z}$ of integers.

We write $O$ for the zero matrix and $I$ for the identity matrix.
We construct the $1251 \times 1839$ block matrix
\[
M = 
\begin{bmatrix} 
C_{1176,1764} & O_{1176,75} \\ 
W_{75,1764} & I_{75} 
\end{bmatrix}
\]
This matrix represents (in arity 5 and multiplicity 4) the operad morphism 
from $\fiveternary$ to the quotient $\mathcal{RB} \cong \mathcal{UB}/\mathcal{I}$.
We refer to columns 1--1764 as the left part and columns 1765--1839 as the right part.
We compute the RCF (row canonical form) and find that $\mathrm{rank}(M) = 1068$.
The RCF has the following block form (omitting zero rows):
\[
\mathrm{RCF}(M) = 
\begin{bmatrix}
X_{1035,1764} & Y_{1035,75} \\
O_{33,1764}   & Z_{33,75}
\end{bmatrix}
\]
Block $X$ is in RCF; it contains those rows of $\mathrm{RCF}(M)$ which have their leading 1s
in the left part.
Block $Y$ contains the right parts of the rows in $X$; this information is not relevant.
The uppermost row of the RCF whose leading 1 appears in the right part
is row 1036 with leading 1 in column 1765.
This gives $1068 - 1035 = 33$ nonzero rows whose leading 1s are in the right part.
Block $Z$ consists of the right parts of these rows; they form a basis for the kernel 
of the rewriting morphism: 
\[
r \colon \fiveternary(5) \longrightarrow \mathcal{UB}(5,4)/\mathcal{J}(5,4) = \mathcal{RB}(5,4).
\]
These rows are the 33 coefficient vectors of the quadratic relations satisfied by
the 5 generating operations.
We find that the entries of $Z$ belong to $\{0,1,-1\}$, and that the number of nonzero
entries in the rows of $Z$ are as follows:
\[
2, 2, 2, 2, 3, 3, 3, 3, 3, 3, 3, 3, 3, 4, 4, 4, 4, 4, 4, 4, 4, 4, 4, 4, 5, 5, 6, 6, 7, 7, 8, 8, 9.
\]
As a measure of the total size of this basis, we use the base 10 logarithm of
the product of these numbers; we obtain $\approx 19.5257$.

We want to find defining relations with the fewest possible terms.
For this we apply the LLL algorithm for lattice basis reduction.
(This method was introduced into the study of polynomial identities by the author and Peresi \cite{BP}.
For an introductory monograph, see the author's book \cite{B}.)
We apply LLL with reduction coefficient $9/10$ (instead of the usual $3/4$) 
to the 33 row vectors, and obtain a smaller new basis (size $\approx 17.4977$),
with the following numbers of nonzero entries:
\[
2, 2, 2, 2, 3, 3, 3, 3, 3, 3, 3, 3, 3, 3, 3, 3, 3, 3, 4, 4, 4, 4, 4, 4, 4, 4, 4, 4, 4, 4, 5, 7, 7.
\]
These rows are the coefficient vectors of the relations stated above.
\end{proof}

\begin{conjecture}
Every relation satisfied by the 5 ternary operations $(-,-,-)_1$, \dots, $(-,-,-)_5$
is a consequence of the 33 quadratic relations of Theorem \ref{maintheorem}.
Hence the Veronese square of the dendriform operad is defined by these 33 relations.
\end{conjecture}



\end{document}